\documentclass[11pt]{article}
\usepackage{epigamath}

\usepackage[notext]{kpfonts}
\usepackage{baskervald}

\setpapertype{A4}

\usepackage[english]{babel}

\usepackage[dvips]{graphicx}     

\title{Pluricomplex Green's functions and Fano manifolds}
\titlemark{Pluricomplex Green's functions and Fano manifolds}
\author{Nicholas McCleerey and Valentino Tosatti}
\authoraddresses{
\authordata{Nicholas McCleerey}{\firstname{Nicholas} \lastname{McCleerey}\\
\institution{Department of Mathematics, Northwestern University, 2033 Sheridan Road, Evanston, IL 60208, USA}\\
\email{njm2@math.northwestern.edu}} \\
\authordata{Valentino Tosatti}{\firstname{Valentino} \lastname{Tosatti}\\
\institution{Department of Mathematics, Northwestern University, 2033 Sheridan Road, Evanston, IL 60208, USA}\\
\email{tosatti@math.northwestern.edu}}
}
\authormark{N. McCleerey and V. Tosatti}
\date{\vspace{-5ex}} 
\journal{\'Epijournal de G\'eom\'etrie Alg\'ebrique} 
\acceptation{Received by the Editors on June 28, 2018. \\ Accepted on May 30, 2019.}

\newcommand{\articlereference}{Volume 3 (2019), Article Nr. 9}

\acknowledgement{Partially supported by NSF RTG grant DMS-1502632 and by NSF grant DMS-1610278.}

\usepackage[all]{xy}

\allowdisplaybreaks

\newdimen\origiwspc
\origiwspc=\fontdimen2\font

\newcommand{\reducespace}[1]{\fontdimen2\font=0.2ex #1 \fontdimen2\font=\origiwspc}

\numberwithin{equation}{section}

\makeatletter

\renewcommand{\p@equation}{\arabic{section}.\arabic{equation}\expandafter\@gobble}
\makeatother

\newcommand{\ddbar}{\sqrt{-1} \partial \overline{\partial}}
\newcommand{\Ric}{\mathrm{Ric}}

\newcommand{\tr}[2]{\textrm{tr}_{#1}{#2}}
\newcommand{\ti}[1]{\tilde{#1}}
\newcommand{\vp}{\varphi}

\newcommand{\ve}{\varepsilon}

\renewcommand{\leq}{\leqslant}
\renewcommand{\geq}{\geqslant}

\newcommand{\be}{\begin{equation}}
\newcommand{\ee}{\end{equation}}

\newtheorem{theorem}{Theorem}[section]
\newtheorem{lemma}[theorem]{Lemma}

\newtheorem{proposition}[theorem]{Proposition}
\newtheorem{tremark}[theorem]{Remark}
\newenvironment{remark}{\begin{tremark} \rm}{\end{tremark}}

\begin{document}


\maketitle



\begin{prelims}


\def\abstractname{Abstract}
\abstract{We show that if a Fano manifold does not admit K\"ahler-Einstein metrics then the K\"ahler potentials along the continuity method subconverge to a function with analytic singularities along a subvariety which solves the homogeneous complex Monge-Amp\`ere equation on its complement, confirming an expectation of Tian-Yau.}



\vspace{0.15cm}

\languagesection{Fran\c{c}ais}{%

\textbf{Titre. Fonctions de Green pluricomplexes et vari\'et\'es de Fano} \commentskip \textbf{R\'esum\'e.} Nous montrons que si une vari\'et\'e de Fano n'admet aucune m\'etrique de K\"ahler-Einstein alors, suivant la m\'ethode de continuit\'e, les potentiels k\"ahl\'eriens  sous-convergent vers une fonction \`a singularit\'es analytiques le long d'une sous-vari\'et\'e, sur le compl\'ementaire de laquelle la fonction est solution de l'\'equation de Monge-Amp\`ere complexe homog\`ene. Cela confirme une attente de Tian-Yau.}

\end{prelims}


\newpage

\setcounter{tocdepth}{1} \tableofcontents

\section{Introduction}
Let $X^n$ be a Fano manifold, i.e. a compact complex manifold with $c_1(X)>0$. A K\"ahler-Einstein metric on $X$ is a K\"ahler metric $\omega$ which satisfies
$$\Ric(\omega)=\omega.$$
This implies that $[\omega]=c_1(X)$. We assume throughout this paper that $X$ does not admit a K\"ahler-Einstein metric. This is known to be equivalent to K-unstability by \cite{CDS} (see also \cite{Ti4}), but we will not use this fact.

We fix a K\"ahler metric $\omega$ with $[\omega]=c_1(X)$, with Ricci potential $F$ defined by $\Ric(\omega)=\omega+\ddbar F$ (normalized by $\int_X (e^F-1)\omega^n=0$). We consider K\"ahler metrics $\omega_t$ with $[\omega_t]=c_1(X)$ which satisfy
$$\Ric(\omega_t)=t\omega_t+(1-t)\omega.$$
We can write $\omega_t=\omega+\ddbar\vp_t$ and the functions $\vp_t$ solve the complex Monge-Amp\`ere equation \cite{Ya}
\begin{equation}\label{ma}
\omega_t^n=e^{F-t\vp_t}\omega^n.
\end{equation}
A solution $\vp_t$ exists on $[0,R(X))$ where $R(X)\leq 1$ is the greatest lower bound for the Ricci curvature of K\"ahler metrics in $c_1(X)$ \cite{Sz}. It is known \cite{Si,Ti} that since $X$ does not admit K\"ahler-Einstein metrics, we must have that $\lim_{t\to R(X)}\sup_X\vp_t=+\infty$. We fix a sequence $t_i\to R(X)$ and write $\vp_i:=\vp_{t_i}$  and $\omega_i:=\omega_{t_i}$. Using this result, together with multiplier ideal sheaves, Nadel \cite[Proposition 4.1]{Na} proved that (up to passing to a subsequence) the measures $\omega_i^n$ converge to zero (as measures) on compact sets of $X\backslash V$ for some proper analytic subvariety $V\subset X$, and in \cite{To2} the second-named author improved this to uniform convergence.

By weak compactness of closed positive currents in a fixed cohomology class, up to subsequences we can extract a limit $\rho$ of $\vp_i-\sup_X\vp_i$ (which may depend on the subsequence), which is an unbounded $\omega$-psh function, and the convergence happens in the $L^1$ topology.

In their work \cite[p.178]{TY}, Tian-Yau expressed the expectation that $\rho$ should have logarithmic poles along a proper analytic subvariety $V\subset X$, and that it should satisfy $(\omega+\ddbar\rho)^n=0$ on $X\backslash V$, so that $\rho$ could be thought of as a kind of pluricomplex Green's function (see also \cite[p.238]{Ti} and \cite[p.109]{Ti2}).

In this note we confirm Tian-Yau's expectation:

\begin{theorem}\label{main} Let $X$ be a Fano manifold without a K\"ahler-Einstein metric, and let $\omega_t=\omega+\ddbar\vp_t$ be the solutions of the continuity method \eqref{ma}.
Given any sequence $t_i\in[0,R(X))$ with $t_i\to R(X)$, choose a subsequence such that $\vp_{t_i}-\sup_{X}\vp_{t_i}$ converge in $L^1(X)$ to an $\omega$-psh function $\rho$.
Then we can find $m\geq 1$ and an $\omega$-psh function $\psi$ on $X$ with analytic singularities
\begin{equation}\label{model}
\psi=\frac{1}{m}\log\sum_{j=1}^p\lambda_j^2|S_j|^2_{h^m},
\end{equation}
for some $\lambda_j\in (0,1]$ and some sections $S_j\in H^0(X,K_X^{-m})$, with nonempty common zero locus $V\subset X$ such that
$\rho-\psi$ is bounded on $X$, and on $X\backslash V$ we have
\begin{equation}\label{zero}
(\omega+\ddbar\rho)^n=0,
\end{equation}
where the Monge-Amp\`ere product is in the sense of Bedford-Taylor \cite{BT}.
\end{theorem}

In particular, Theorem \ref{main} implies that the non-pluripolar Monge-Amp\`ere operator of $\rho$ (defined in \cite{BEGZ}) vanishes identically on $X$. On the other hand, there is another meaningful Monge-Amp\`ere operator that can be applied to $\rho$. Indeed, the fact that $\rho-\psi\in L^\infty(X)$ implies that $\rho$ itself has analytic singularities. In \cite{ABW} Andersson-B\l ocki-Wulcan defined a Monge-Amp\`ere operator for $\omega$-psh functions with analytic singularities (generalizing earlier work of Andersson-Wulcan \cite{AW} in the local setting). In general, applying this Monge-Amp\`ere operator to $\rho$ will produce a Radon measure $\mu$ on $X$ (which may be identically zero in some cases), which by Theorem \ref{main} is supported on the analytic set $V$, thus providing geometrically interesting examples of unbounded quasi-psh functions on compact K\"ahler manifolds with Monge-Amp\`ere operator concentrated on a subvariety (see also \cite{ACCP,ACP} for related results in the local setting). In particular, this answers \cite[Question 1 (c)]{MA}, an open problem raised at the AIM workshop ``The complex Monge-Amp\`ere equation'' in August 2016 (cf. the related \cite[Question 12]{DGZ}).

Note that in general a formula for the total mass of $\mu$ is proved in \cite[Theorem 1.2]{ABW}, and it satisfies
$$\int_X\mu\leq \int_X\omega^n,$$
with strict inequality in general (but it is not hard to see that if $\dim X=2$ and $V$ is a finite set then equality holds). Therefore, the measure $\mu$ is in general different from the measures that one obtains as weak limits of $(\omega+\ddbar\vp_i)^n$ (up to subsequences), whose total mass is always equal to $\int_X\omega^n$.

\begin{remark}[Remark added in proof] After this work was posted on the arXiv, and partly prompted by it, B\l ocki \cite{Bl} modified the definition of the Monge-Amp\`ere operator for $\omega$-psh functions with analytic singularities of \cite{AW,ABW}, and with his definition the total mass is always equal to $\int_X\omega^n$. It is an interesting question to determine whether this Monge-Amp\`ere operator equals the weak limit of $(\omega+\ddbar\vp_i)^n$.
\end{remark}

\begin{remark}
As in the second-named author's previous work \cite{To2}, Theorem \ref{main} has a direct counterpart for solutions of the normalized K\"ahler-Ricci flow, instead of the continuity method \eqref{ma}. The statement is identical to Theorem \ref{main}, except that now the sequence $t_i$ goes to $+\infty$. The proof is also almost verbatim the same, and the partial $C^0$ estimate along the flow is proved in \cite{CW,CW2} (see also \cite{Ba}). All other ingredients used also have well-known counterparts for the flow (see \cite{To2}). We leave the simple details to the interested reader.
\end{remark}

\begin{remark}
The behavior of the solutions $\omega_t$ of \eqref{ma} as $t\to R(X)$ has been investigated in the past. If the manifold is K-stable, \cite{DSz} show that $\omega_t$ converge smoothly to a K\"ahler-Einstein metric. If on the other hand no such metric exists, the blowup behavior of $\omega_t$ has been investigated in \cite{Na,To2} in the setting of this paper, and also in \cite{DSz,Li,SZ} by allowing reparametrizations of the metrics by diffeomorphisms.
\end{remark}

The proof of Theorem \ref{main} relies on the partial $C^0$ estimate for solutions of \eqref{ma} which was established by Sz\'ekelyhidi \cite{Sz2}. We recall this in section \ref{sectp}, together with a well-known reformulation of this estimate (Proposition \ref{partc0}). In section \ref{sects} we observe that this gives us the singularity model function $\psi$ in \eqref{model}, and it also implies that $\rho$ has the same singularity type as $\psi$. In section \ref{sectc} we show the general fact that every $\omega$-psh function on $X$ with the same singularity type as $\psi$ has vanishing Monge-Amp\`ere operator outside $V$, thus proving Theorem \ref{main}. This relies on a geometric understanding of the rational map defined by the sections $\{S_j\}$ as in Theorem \ref{main}. Lastly, in section \ref{sectg} we discuss the pluricomplex Green's function with the same singularity type as $\psi$.

\bigskip

{\bf Acknowledgments. } We thank Z. B\l ocki, T.C. Collins, S. Ko\l odziej and D.H. Phong for discussions on this topic at the AIM workshop ``The complex Monge-Amp\`ere equation'' in August 2016, E. Wulcan for her interest in this work and A. Rashkovskii and the referee for useful comments. The second-named author is also grateful to S.-T. Yau for related discussions over the years. This work was completed during the second-named author's visit to the Institut Henri Poincar\'e in Paris (supported by a Chaire Poincar\'e funded by the Clay Mathematics Institute) which he would like to thank for the hospitality and support.

\section{The partial $C^0$ estimate}\label{sectp}
To start we fix some notation. We choose a Hermitian metric $h$ on $K_X^{-1}$ with curvature $R_h=\omega$ (such $h$ is unique up to scaling), and let $h^m$ be the induced metric on $K_X^{-m}$, for all $m\geq 1$. Let $N_m=\dim H^0(X,K_X^{-m})$, and for any $m\geq 1$ define the density of states function
$$\rho_m(\omega)=\sum_{j=1}^{N_m} |S_j|^2_{h^m},$$
where $S_1,\dots,S_{N_m}$ are a basis of $H^0(X,K_X^{-m})$ which is orthonormal with respect to the $L^2$ inner product $\int_X\langle S_1,S_2\rangle_{h^m}\omega^n$. Clearly $\rho_m(\omega)$ is independent of the choice of basis, and is also unchanged if we scale $h$ by a constant. The integral  $\int_X \rho_m(\omega)\omega^n$ equals $N_m$, and if $m$ is sufficiently large so that $K_X^{-m}$ is very ample, then $\rho_m(\omega)$ is strictly positive on $X$.
If we apply this same construction to the metrics $\omega_t$ and Hermitian metrics $h_t=he^{-\vp_t}$ we get a density of states function $\rho_m(\omega_t).$
Following \cite{Ti2}, we say that a ``partial $C^0$ estimate'' holds if there exist $m\geq 1$ and a constant $C>0$ such that
\begin{equation}\label{partialc0}
\inf_{X}\rho_{m}(\omega_t)\geq C^{-1},
\end{equation}
holds for all $t\in [0,R(X)).$ The reason for this name is explained by the following proposition, which is essentially well-known (see \cite[Lemma 2.2]{Ti2} and \cite[Proposition 5.1]{To}),
but we provide the details for convenience:

\begin{proposition}\label{partc0}
If a partial $C^0$ estimate holds then there exists $m\geq 1$, such that for all $\ve>0$ we can find a constant $C>0$  so that for all $t\in[\ve,R(X))$ we can find real numbers $1=\lambda_1(t)\geq\dots\geq\lambda_{N_m}(t)>0$ and a basis $\{S_j(t)\}_{1\leq j\leq N_m}$ of $H^0(X,K_{X}^{-m})$, orthonormal with respect to the $L^2$ inner product of $\omega, h^m$, such that for all $t\in[0,R(X))$ we have
\begin{equation}\label{partiall}
\sup_{X}\left|\vp_t-\sup_{X}\vp_t-\frac{1}{m}\log\sum_{j=1}^{N_m} \lambda_j(t)^2|S_j(t)|^2_{h^m}\right|\leq C.
\end{equation}
\end{proposition}

In the rest of the paper we will fix a value of $\ve>0$ once and for all, for example $\ve=R(X)/2$. The precise choice is irrelevant, since we are only interested in the behavior as $t\to R(X)$.
\begin{proof}
First, it is well-known that for all $m\geq 1$ and $\ve>0$ there is a constant $C$ such that for all $t\in [\ve,R(X))$ we have
\begin{equation}\label{upperbd}
\rho_m(\omega_t)\leq C.
\end{equation}
To see this, first observe for every $S\in H^0(X,K_X^{-m})$ we have
\begin{equation}\label{moser}
\Delta_{\omega_t}|S|^2_{h_t^m}=|\nabla S|^2_t -2m|S|^2_{h_t^m}\geq -2m|S|^2_{h_t^m},
\end{equation}
and that since $\Ric(\omega_t)\geq t\omega_t\geq \ve\omega_t$, Myers' Theorem gives a uniform upper bound for $\mathrm{diam}(X,\omega_t)$ and then Croke \cite{Cr} and Li \cite{LiP} show that the Sobolev constant of $(X,\omega_t)$ has a uniform upper bound. We can then apply Moser iteration to \eqref{moser} to get
\begin{equation}\label{supbd}
\sup_{X}|S|^2_{h_t^m}\leq C\int_{X}|S|^2_{h_t^m}\omega_t^n\leq C,
\end{equation}
provided we assume that $\int_{X}|S|^2_{h_t^m}\omega_t^n=1$. Taking now an orthonormal basis of sections and summing we obtain \eqref{upperbd}.

Thanks to \eqref{upperbd} we know that for $t\in [\ve,R(X))$ a partial $C^0$ estimate is equivalent to
\begin{equation}\label{mos1}
\sup_{X}|\log\rho_m(\omega_t)|\leq C.
\end{equation}
We now take a basis $\{\ti{S}_j(t)\}_{1\leq j\leq N_m}$ of $H^0(X,K_{X}^{-m})$ orthonormal with respect to the $L^2$ inner product of $\omega_t, h_t^m$ and notice that since $h_t^m=e^{-m\vp_t}h^m$ we clearly have
$$\vp_t=\frac{1}{m}\log\frac{\sum_{j=1}^{N_m}|\ti{S}_j(t)|^2_{h^m}}{\sum_{j=1}^{N_m}|\ti{S}_j(t)|^2_{h^m_t}},$$
which is equivalent to
\begin{equation}\label{mos2}
\vp_t-\frac{1}{m}\log \sum_{j=1}^{N_m}|\ti{S}_j(t)|^2_{h^m}=-\frac{1}{m}\log\rho_m(\omega_t).
\end{equation}
It follows from \eqref{mos1} and \eqref{mos2} that that for $t\in [\ve,R(X))$ a partial $C^0$ estimate is equivalent to an estimate
$$\sup_{X}\left|\vp_t-\frac{1}{m}\log \sum_{j=1}^{N_m}|\ti{S}_j(t)|^2_{h^m} \right|\leq C.$$
We now choose another basis $\{S_j\}_{1\leq j\leq N_m}$ of $H^0(X,K_{X}^{-m})$ orthonormal with respect to the $L^2$ inner product of $\omega, h^m$. After modifying $S_j$ and $\ti{S}_j(t)$ by $t$-dependent unitary transformations, we obtain orthonormal bases $\{S_j(t)\}_{1\leq j\leq N_m}$ with respect to $\omega,h^m$, and $\{\ti{S}_j(t)\}_{1\leq j\leq N_m}$ with respect to $\omega_t, h_t^m$ such that
$$\ti{S}_j(t)=\mu_j(t) S_j(t),$$
for some positive real numbers $\mu_j(t)$, with $\mu_1(t)\geq\dots\geq \mu_{N_m}(t)>0$.
We then let $\lambda_j(t)=\mu_j(t)/\mu_{1}(t)$ and we see that a partial $C^0$ estimate is equivalent to
\begin{equation}\label{part1}
\sup_{X}\left|\vp_t-\frac{2}{m}\log\mu_{1}(t)-\frac{1}{m}\log \sum_{j=1}^{N_m}\lambda_j(t)^2|S_j(t)|^2_{h^m} \right|\leq C.
\end{equation}
We now claim that if a partial $C^0$ estimate holds, then for all $t\in[\ve,R(X))$ we also have
\begin{equation}\label{part2}
\left|\frac{2}{m}\log\mu_{1}(t)-\sup_{X}\vp_t \right|\leq C.
\end{equation}
Once this is proved, combining \eqref{part1} and \eqref{part2} we get \eqref{partiall}. To prove \eqref{part2}, first use \eqref{supbd} to get
$$C\geq \sup_{X}|\ti{S}_{1}(t)|^2_{h_t^m}\geq\mu_{1}(t)^2 \sup_{X}|S_{1}(t)|^2_{h^m}e^{-m\sup_{X}\vp_t},$$
and the fact $\int_{X}|S_{1}(t)|^2_{h^m}\omega^n=1$
implies that $\sup_{X}|S_{1}(t)|^2_{h^m}\geq 1/\mathrm{Vol}(X,\omega),$
and so
$$\left(\frac{2}{m}\log\mu_{1}(t)-\sup_{X}\vp_t\right)\leq C.$$
On the other hand the partial $C^0$ estimate \eqref{partialc0} implies that
\begin{equation}\label{part3}
C^{-1}\leq \rho_{m}(\omega_t)=\sum_{j=1}^{N_m}|\ti{S}_j(t)|^2_{h_t^m}\leq
\mu_{1}(t)^2\sum_{j=1}^{N_m}|S_j(t)|^2_{h^m}e^{-m\vp_t},
\end{equation}
and we clearly have that
\begin{equation}\label{c0norm}
\sup_{j} \sup_{X}|S_j(t)|^2_{h^m}\leq C,
\end{equation}
since the sections $\{S_j(t)\}$ are just varying in a compact unitary group (or one can also repeat the Moser iteration argument of \eqref{upperbd} for the fixed metric $\omega$).
This together with \eqref{part3}, evaluated at the point where $\vp_t$ achieves its maximum,
gives the reverse inequality
$$\left(\sup_{X}\vp_t-\frac{2}{m}\log\mu_{1}(t)\right)\leq C,$$
which completes the proof of \eqref{part2}.
\end{proof}

\section{The singularity model function}\label{sects}
The next goal is to use the partial $C^0$ estimate in Proposition \ref{partc0} to construct a singular $\omega$-psh function $\psi$ which will have the same singularity type of any weak limit of the normalized solutions $\vp_i-\sup_X\vp_i$ of the continuity method.

Let the notation be as in Proposition \ref{partc0}, and in particular we fix once and for all a value of $m\geq 1$ given there. We can find a sequence $t_i\to R(X)$ and an $\omega$-psh function $\rho$ with $\sup_X\rho=0$ such that $\vp_i-\sup_X\vp_i\to \rho$ in $L^1(X)$, and pointwise a.e. Passing to a subsequence, we can find a basis $\{S_j\}_{1\leq j\leq N_m}$ of $H^0(X,K_{X}^{-m})$ orthonormal with respect to the $L^2$ inner product of $\omega, h^m$, such that
$S_j(t_i)\to S_j$ smoothly as $i\to\infty$, for all $1\leq j\leq N_m$. The change of basis matrix from $\{S_j\}_{1\leq j\leq N_m}$ to $\{S_j(t)\}_{1\leq j\leq N_m}$ induces an automorphism $\sigma(t)$ of $\mathbb{CP}^{N_m-1}$, such that $\sigma(t_i)\to\mathrm{Id}$ smoothly as $i\to\infty$.

For ease of notation, write
$$\psi_t=\frac{1}{m}\log\sum_{j=1}^{N_m} \lambda_j(t)^2|S_j(t)|^2_{h^m}.$$
These functions are K\"ahler potentials for $\omega$ since
\begin{equation}\label{psh}
\omega+\ddbar\psi_t=\frac{\iota^*\sigma(t)^*\tau(t)^*\omega_{FS}}{m}>0,
\end{equation}
where $\iota:X\hookrightarrow\mathbb{CP}^{N_m-1}$ is the Kodaira embedding map given by the sections $\{S_j\}_{1\leq j\leq N_m},$
the map $\tau(t)$ is the automorphism of $\mathbb{CP}^{N_m-1}$ induced by the diagonal matrix with entries $\{\lambda_j(t)\}_{1\leq j\leq N_m}$, 
and $\omega_{FS}$ is the Fubini-Study metric on $\mathbb{CP}^{N_m-1}$. The identity in \eqref{psh} follows directly from the definition of the Fubini-Study metric $\omega_{FS}$ on $\mathbb{CP}^{N_m-1}$, which on $\mathbb{C}^{N_m}\backslash\{0\}$ is given explicitly by $\omega_{FS}=\ddbar\log\sum_{j=1}^{N_m}|z_j|^2,$ and from the fact that the curvature of $h$ is $\omega$.

Up to passing to a subsequence of $t_i$, we may assume that $\lambda_j(t_i)\to \lambda_j$ as $i\to\infty$ for all $j$, and we have
$$1=\lambda_1\geq \dots\geq \lambda_p>0=\lambda_{p+1}=\dots=\lambda_{N_m},$$
for some $1\leq p<N_m$. The case $p=N_m$ is impossible because by \eqref{partiall} it would imply a uniform $L^\infty$ bound for $\vp_t$ and so $X$ would admit a K\"ahler-Einstein metric. For the same reason, the set $V:=\{S_1=\cdots=S_p=0\}$ must be a nonempty proper analytic subvariety of $X$.

Note that thanks to \eqref{partiall} we can write
$$\omega_t^n=e^{F-t(\vp_t-\sup_X\vp_t)} e^{-t\sup_X\vp_t}\omega^n\leq Ce^{t\psi_t}e^{-t\sup_X\vp_t}\omega^n,$$
and since the term $e^{t\psi_t}$ is uniformly bounded on compact sets of $X\backslash V$, we see immediately that
\begin{equation}\label{tozero}
\omega_t^n\to 0,
\end{equation}
uniformly on compact sets of $X\backslash V$ (this result was proved in \cite{To2} without using the partial $C^0$ estimate, which was not available at the time, with weaker results established earlier in \cite{Na}).

Let then $$\psi=\frac{1}{m}\log\sum_{j=1}^p\lambda_j^2|S_j|^2_{h^m},$$
which is a smooth function on $X\backslash V$ which approaches $-\infty$ uniformly on $V$. Since $e^{m\psi_t}\to e^{m\psi}$ smoothly on $X$, and since $\psi_t$ are smooth and $\omega$-psh, it follows that $\psi$ is $\omega$-psh. This will be our singularity model function in the rest of the argument, as we now explain:

\begin{lemma}\label{parr}
Define the class
$$\mathcal{C}=\{\eta\in PSH(X,\omega)\ |\ \eta-\psi\in L^\infty(X)\},$$
of $\omega$-psh functions with the same singularity type as $\psi$.
Then we have that $\rho\in\mathcal{C}$.
\end{lemma}
\begin{proof}
Recall that we have $\vp_i-\sup_X\vp_i\to\rho$ a.e. on $X$.
Thanks to \eqref{partiall}, the function $\rho$ satisfies
\begin{equation}\label{est}
|\rho-\psi|\leq C,
\end{equation}
a.e. on $X$, which implies the same inequality on all of $X$ by elementary properties of psh functions (cf. \cite[Theorem K.15]{Gu}), thus showing that $\rho\in\mathcal{C}$.
\end{proof}

\section{Understanding the class $\mathcal{C}$}\label{sectc}
We now exploit the geometry of our setting to gain a better understanding of the class of functions $\mathcal{C}$.

The sections $\{\lambda_jS_j\}_{1\leq j\leq p}$ define a rational map $\Phi:X\dashrightarrow\mathbb{CP}^{p-1}$, with indeterminacy locus $Z\subset V$ (this inclusion is in general strict, since $\mathrm{codim} Z\geq 2$ while $V$ may contain divisorial components). Let $Y$ be the image of $\Phi$, i.e. the closure of $\Phi(X\backslash Z)$ in $\mathbb{CP}^{p-1}$, which is an irreducible projective variety. By resolving the indeterminacies of $\Phi$ we get a modification $\mu:\ti{X}\to X$, obtained as a sequence of blowups with smooth centers, and a holomorphic map $\Psi:\ti{X}\to Y$ such that $\Psi=\Phi\circ\mu$ holds on $\ti{X}\backslash \mu^{-1}(Z)$. We may also assume without loss of generality that $\mu$ principalizes the ideal sheaf generated by $\{S_j\}_{1\leq j\leq p}$, so that we have
$$\mu^*(\omega+\ddbar\psi)=\theta+[E],$$
where $E$ is an effective $\mathbb{R}$-divisor with $\mu(E)\subset V$, and $\theta$ is a smooth closed semipositive $(1,1)$ form on $\ti{X}$. We will denote by $\omega_{FS,p}$ the Fubini-Study metric on $\mathbb{CP}^{p-1}$. To identify $\theta$, note that on $X\backslash V$ we have by definition $\omega+\ddbar\psi=\frac{\Phi^*\omega_{FS,p}}{m},$ and so on $\ti{X}\backslash \mu^{-1}(V)$ we have
$$\mu^*(\omega+\ddbar\psi)=\frac{\mu^*\Phi^*\omega_{FS,p}}{m}=\frac{\Psi^*\omega_{FS,p}}{m},$$
and so $\theta=\frac{\Psi^*\omega_{FS,p}}{m}$ on $\ti{X}\backslash \mu^{-1}(V)$, and hence everywhere since both sides of this equality are smooth forms on all of $\ti{X}$. This proves the key relation
\begin{equation}\label{key}
\mu^*(\omega+\ddbar\psi)=\frac{\Psi^*\omega_{FS,p}}{m}+[E].
\end{equation}

Let $\ti{X}\overset{\nu}{\to}\ti{Y}\overset{q}{\to}Y$ be the Stein factorization of $\Psi$, where $\ti{Y}$ is an irreducible projective variety, the map $\nu$ has connected fibers, and $q$ is a finite morphism.

We have that $q^*\omega_{FS,p}$ is a smooth semipositive $(1,1)$ form on $\ti{Y}$, in the sense of analytic spaces. Since $\nu$ has compact connected fibers, a standard argument shows that the set of $\frac{\Psi^*\omega_{FS,p}}{m}$-psh functions on $\ti{X}$ can be identified with the set of (weakly) $\frac{q^*\omega_{FS,p}}{m}$-psh functions on $\ti{Y}$ via $\nu^*$ (indeed the restriction of every $\frac{\Psi^*\omega_{FS,p}}{m}$-psh function to any fiber of $\nu$ is plurisubharmonic and hence constant on that fiber). We will use this standard argument several other times in the following.

Here and in the following, as in \cite{De}, a weakly quasi-psh function on a compact analytic space means a quasi-psh function on its regular part which is locally bounded above near the singular set. As shown in \cite[\S 1]{De}, weakly quasi-psh functions are the same as usual quasi-psh functions if the analytic space is normal, and otherwise they can be identified with quasi-psh functions on its normalization.

\begin{proposition}\label{corr}
Given any function $\eta\in\mathcal{C}$, there is a unique bounded weakly $\frac{q^*\omega_{FS,p}}{m}$-psh function $u$ on $\ti{Y}$ such that
\begin{equation}\label{rel}
\mu^*\eta=\mu^*\psi+\nu^*u.
\end{equation}
Conversely, given any bounded weakly $\frac{q^*\omega_{FS,p}}{m}$-psh function $u$ on $\ti{Y}$ there is a unique function $\eta\in\mathcal{C}$ such that \eqref{rel} holds.
\end{proposition}

The relation in \eqref{rel} thus allows us to identify the class $\mathcal{C}$ with the class of bounded weakly $\frac{q^*\omega_{FS,p}}{m}$-psh functions on $\ti{Y}$.

Next, we observe that
\begin{proposition}\label{dim}
We have that
$$\dim Y<\dim X.$$
\end{proposition}
This is a consequence of our assumption that $X$ does not admit a K\"ahler-Einstein metric.

Lastly, every function $\eta\in\mathcal{C}$ belongs to $L^\infty_{\rm loc}(X\backslash V)$, and so its Monge-Amp\`ere operator $(\omega+\ddbar\eta)^n$ is well-defined on $X\backslash V$ thanks to Bedford-Taylor \cite{BT}. Combining the results in Propositions \ref{corr} and \ref{dim} we will obtain:

\begin{theorem}\label{zeroma}
For every $\eta\in\mathcal{C}$ we have that
$$(\omega+\ddbar\eta)^n=0,$$
on $X\backslash V$.
\end{theorem}

In particular, this holds for the function $\rho$, thanks to Lemma \ref{parr}, and 
Theorem \ref{main} thus follows from these.

\begin{proof}[Proof of Proposition \ref{corr}]
If $\eta$ is an $\omega$-psh function on $X$ with $\eta-\psi\in L^\infty(X)$, i.e. $\eta$ is an element of $\mathcal{C}$, then using \eqref{key} we can write
$$\mu^*(\omega+\ddbar\eta)=\frac{\Psi^*\omega_{FS,p}}{m}+\ddbar\mu^*(\eta-\psi)+[E],$$
where $E$ is as in \eqref{key} and $\mu^*(\eta-\psi)\in L^\infty(\ti{X})$. Applying the Siu decomposition, we see that
$$\frac{\Psi^*\omega_{FS,p}}{m}+\ddbar\mu^*(\eta-\psi)\geq 0,$$
weakly, and so $$\mu^*(\eta-\psi)=\nu^*u_\eta,$$
for a bounded weakly $\frac{q^*\omega_{FS,p}}{m}$-psh functions $u_\eta$ on $\ti{Y}$, which is uniquely determined by $\eta$ (and $\psi$, which we view as fixed).

Conversely, given a bounded weakly $\frac{q^*\omega_{FS,p}}{m}$-psh function $u$ on $\ti{Y}$, we have that $\nu^*u$ is $\frac{\Psi^*\omega_{FS,p}}{m}$-psh and bounded on $\ti{X}$ and so
$$0\leq \frac{\Psi^*\omega_{FS,p}}{m}+[E]+\ddbar \nu^*u=\mu^*\omega+\ddbar(\mu^*\psi+\nu^*u),$$
and so $\mu^*\psi+\nu^*u$ descends to an $\omega$-psh function $\eta_u$ on $X$ with $\eta_u-\psi\in L^\infty(X)$, i.e. $\eta_u\in\mathcal{C}$.

These two constructions are inverses to each other, and so we obtain the desired bijective correspondence between functions in $\mathcal{C}$ and bounded weakly $\frac{q^*\omega_{FS,p}}{m}$-psh functions on $\ti{Y}$.
\end{proof}

\begin{proof}[Proof of Proposition \ref{dim}]
On $X$ we have the estimate
\begin{equation}\label{sc}
\omega_t \geq C^{-1} \frac{\iota^*\sigma(t)^*\tau(t)^*\omega_{FS}}{m},
\end{equation}
which is a direct consequence of the partial $C^0$ estimate (see e.g. \cite[Lemma 4.2]{DS}). We can also give a direct proof by calculating
\[
\Delta_{\omega_t}\left(\log\tr{\omega_t}{\left(\frac{\iota^*\sigma(t)^*\tau(t)^*\omega_{FS}}{m}\right)}-A(\vp_t-\sup_X\vp_t-\psi_t)\right)\geq \tr{\omega_t}{\left(\frac{\iota^*\sigma(t)^*\tau(t)^*\omega_{FS}}{m}\right)}-C,
\]
if $A$ is sufficiently large, and applying the maximum principle together with the partial $C^0$ estimate \eqref{partiall} (for this calculation we used that the bisectional curvature of the metrics $\frac{\iota^*\sigma(t)^*\tau(t)^*\omega_{FS}}{m}$ have a uniform upper bound independent of $t$).

If we had $\dim Y=\dim X$ then the rational map $\Phi$ would be generically finite, so there would be a nonempty open subset $U\Subset X\backslash V$ such that $\Phi|_U$ is a biholomorphism with its image.
Recall that $\Phi$ is the rational map defined by the sections $\{\lambda_jS_j\}_{1\leq j\leq p}$, while $\iota:X\hookrightarrow\mathbb{CP}^{N_m-1}$ is the embedding defined by
the sections $\{S_j\}_{1\leq j\leq N_m}$, and so $\Phi=\ti{\tau}\circ P\circ \iota$ where $P:\mathbb{CP}^{N_m-1}\dashrightarrow \mathbb{CP}^{p-1}$ is the linear projection given by $[z_1:\cdots:z_{N_m}]\mapsto[z_1:\cdots:z_p]$ and $\ti{\tau}:\mathbb{CP}^{p-1}\to\mathbb{CP}^{p-1}$ is the automorphism given by
$$[z_1:\cdots:z_p]\mapsto[\lambda_1z_1:\cdots:\lambda_pz_p].$$
In particular, on the embedded open $n$-fold $\iota(U)$, we have that $P|_{\iota(U)}$ is also a biholomorphism with its image. The automorphisms $\tau(t_i)$ descend to automorphisms $\ti{\tau}(t_i)$ on $\mathbb{CP}^{p-1}$, and now as $i\to\infty$ these converge smoothly to the automorphism $\ti{\tau}$. Thus $P\circ\tau(t_i)\circ\sigma(t_i)\circ\iota=\ti{\tau}(t_i)\circ P\circ\sigma(t_i)\circ \iota$, which converge smoothly as maps to $\ti{\tau}\circ P\circ\iota=\Phi$ on $U$ as $i\to\infty$.

Since $\Phi$ is an isomorphism on $U$, smooth convergence gives us that $P\circ\tau(t_i) \circ\sigma(t_i)\circ \iota$ is a local isomorphism. Thus, after possibly shrinking $U$,
$$P:(\tau(t_i)\circ\sigma(t_i)\circ\iota)(U)\to (P\circ \tau(t_i)\circ\sigma(t_i)\circ\iota)(U)=(\ti{\tau}(t_i)\circ P\circ\sigma(t_i)\circ \iota)(U)$$
is an isomorphism, and for $i$ large the open sets $(\ti{\tau}(t_i)\circ P\circ\sigma(t_i)\circ \iota)(U)\subset\mathbb{CP}^{p-1}$ converge to the open set $(\ti{\tau}\circ P\circ \iota)(U)$ in the Hausdorff sense. Up to shrinking $U$, there is an open subset $V\subset\mathbb{CP}^{p-1}$ that contains $(\ti{\tau}(t_i)\circ P\circ\sigma(t_i)\circ \iota)(U)$ for all $i$ large, and still $P^{-1}$ is well-defined on $V$ (and $P:P^{-1}(V)\to V$ is a biholomorphism), so that $P^{-1}(V)$ contains $(\tau(t_i)\circ\sigma(t_i)\circ\iota)(U)$ for all $i$ large, and on $P^{-1}(V)$ we have
\begin{equation}\label{sc1}
P^*\omega_{FS,p}\leq C\omega_{FS},
\end{equation}
On $U$ we also have that $\frac{\iota^*\sigma(t_i)^*\tau(t_i)^*P^*\omega_{FS,p}}{m}$ converges smoothly to $\frac{\Phi^*\omega_{FS,p}}{m}$, which is a K\"ahler metric on $U$.
Thanks to \eqref{sc} and \eqref{sc1}, on $U$ we have
$$\omega_i \geq C^{-1} \frac{\iota^*\sigma(t_i)^*\tau(t_i)^*\omega_{FS}}{m} \geq C^{-1} \frac{\iota^*\sigma(t_i)^*\tau(t_i)^*P^*\omega_{FS,p}}{m}\geq C^{-1}\frac{\Phi^*\omega_{FS,p}}{m},$$
for all $i$ large, which implies that $\int_U \omega_i^n \geq C^{-1}$, which is absurd thanks to \eqref{tozero}.
\end{proof}
\begin{remark}\label{uniq}
In particular we see that if $\dim Y=0$ (i.e. $Y$ is a point) then we have $\mathcal{C}=\{\psi+s\}_{s\in \mathbb{R}}$. On the other hand as long as $\dim Y>0$ the class $\mathcal{C}$ is always rather large.
\end{remark}

\begin{proof}[Proof of Theorem \ref{zeroma}]
Thanks to Proposition \ref{corr}, every $\eta\in\mathcal{C}$ satisfies $\mu^*\eta=\mu^*\psi+\nu^*u$ for some bounded weakly $\frac{q^*\omega_{FS,p}}{m}$-psh function $u$ on $\ti{Y}$. Then using \eqref{key} we have
\[\begin{split}
\mu^*(\omega+\ddbar\eta)&=\frac{\Psi^*\omega_{FS,p}}{m}+\ddbar\nu^*u+[E]\\
&=\nu^*\left(\frac{q^*\omega_{FS,p}}{m}+\ddbar u\right)+[E],
\end{split}\]
and so if $K$ is any compact subset of $X\backslash V$, since $\mu$ is an isomorphism on $\mu^{-1}(K)$, we get
\[\begin{split}
\int_K(\omega+\ddbar\eta)^n&=\int_{\mu^{-1}(K)}\mu^*(\omega+\ddbar\eta)^n\\
&=\int_{\mu^{-1}(K)}\nu^*\left(\frac{q^*\omega_{FS,p}}{m}+\ddbar u\right)^n=0,
\end{split}\]
since $\dim \ti{Y}=\dim Y<\dim X$ by Proposition \ref{dim}.
\end{proof}

\section{The pluricomplex Green's function}\label{sectg}
We can also consider the pluricomplex Green's function with singularity type determined by $\psi$, namely
\begin{equation}\label{en2}
G=\sup\{u\ |\ u\in PSH(X,\omega), u\leq 0, u\leq \psi+O(1)\}^*,
\end{equation}
which is the compact manifold analog of the construction in \cite{RS}, and has been studied in detail in \cite{CG,PS,RWN2} and references therein. In particular, since $\psi$ has analytic singularities, it follows from \cite{RS,RWN2} that $G\in\mathcal{C}$.

Thanks to Proposition \ref{corr} we can write
\begin{equation}\label{rel2}
\mu^*G=\mu^*\psi+\nu^*F,
\end{equation}
for a bounded weakly $\frac{q^*\omega_{FS,p}}{m}$-psh function $F$ on $\ti{Y}$. The function $F$ is itself given by a suitable envelope.
\begin{proposition}\label{prop51}
The pluricomplex Green's function $G$ satisfies \eqref{rel2} where $F$ is the envelope on $\ti{Y}$ given by
\begin{equation}\label{en1}
F=\sup\{w\ |\ w\in PSH(\ti{Y},q^*\omega_{FS,p}/m), w\leq -\nu_*\mu^*\psi\}^*,
\end{equation}
and where we are writing
$$\nu_*(f)(y)=\sup_{x\in\nu^{-1}(y)}f(x),$$
for any function $f$ on $\ti{X}, y\in \ti{Y}$.
\end{proposition}

In other words, $F$ is given by a quasi-psh envelope with obstacle $-\nu_*\mu^*\psi$ on $\ti{Y}$.
\begin{proof}
Write $E=\sum_i\lambda_i E_i$ for $E_i$ prime divisors and $\lambda_i\in\mathbb{R}_{>0}$, and for each $i$ fix a defining section $s_i$ of $\mathcal{O}(E_i)$ and a smooth metric $h_i$ on $\mathcal{O}(E_i)$ with curvature form $R_i$. For brevity, we will write $|s|^2_h=\prod_i |s_i|^{2\lambda_i}_{h_i}$ and $R_h=\sum_i \lambda_i R_i$.
Then the Poincar\'e-Lelong formula gives
$$[E]=\ddbar\log|s|^2_h+R_h,$$
and we obtain that $\mu^*\omega-R_h$ is cohomologous to $\frac{\Psi^*\omega_{FS,p}}{m}$ and
$$\mu^*\omega-R_h=\frac{\Psi^*\omega_{FS,p}}{m}+\ddbar(\log|s|^2_h-\mu^*\psi),$$
and $\mu^*\psi-\log|s|^2_h$ is smooth on all of $\ti{X}$. Note that if we denote by
$$\ti{G}=\sup\{u\ |\ u\in PSH(\ti{X},\mu^*\omega), u\leq 0, u\leq \log|s|^2_h+O(1)\}^*,$$
then we have that $\ti{G}=\mu^*G$ (this is again because every $\mu^*\omega$-psh function on $\ti{X}$ is in fact the pullback of an $\omega$-psh function on $X$).

As in \cite{McC}, we use a trick from \cite[Section 4]{Be} (see also \cite{RS}), to show that
$$\ti{G}=\log|s|^2_h+\sup\{v\ |\ v\in PSH(\ti{X},\mu^*\omega-R_h), v\leq -\log|s|^2_h\}^*.$$
For the reader's convenience, we supply the simple proof.
Denote the right hand side by $\hat{G}$. For one direction, if $v$ is $(\mu^*\omega-R_h)$-psh and satisfies $v\leq -\log|s|^2_h$, then
$u:=v+\log|s|^2_h$ satisfies $u\leq 0$ but also since $v\leq C$ on $\ti{X}$, we see that $u\leq \log|s|^2_h+C$, and also
\[\begin{split}
\mu^*\omega+\ddbar u&=\mu^*\omega+\ddbar\log|s|^2_h+\ddbar v\\
&=\mu^*\omega-R_h+[E]+\ddbar v\\
&\geq \mu^*\omega-R_h+\ddbar v\geq 0,
\end{split}\]
and so $\hat{G}\leq\ti{G}$. Conversely, if $u$ is $\mu^*\omega$-psh and satisfies $u\leq 0$ and $u\leq \log|s|^2_h+C$ for some $C$, then the Siu decomposition of $\mu^*\omega+\ddbar u$ contains $[E]$ and so
$$0\leq \mu^*\omega+\ddbar u-[E]=\mu^*\omega-R_h+\ddbar(u-\log|s|^2_h),$$
and so $v:=u-\log|s|^2_h$ is $(\mu^*\omega-R_h)$-psh and satisfies $v\leq -\log|s|^2_h,$ and it follows that $\ti{G}\leq \hat{G}$, which proves our claim.

But finally note that for all $x\in\ti{X}$ we have
\[\begin{split}
&\log|s|^2_h(x)+\sup\{v(x)\ |\ v\in PSH(\ti{X},\mu^*\omega-R_h), v\leq -\log|s|^2_h\}\\
&=\mu^*\psi(x)+\sup\{v(x)\ |\ v\in PSH(\ti{X},\Psi^*\omega_{FS,p}/m), v\leq -\mu^*\psi\}\\
&=\mu^*\psi(x)+\sup\{w(\nu(x))\ |\ w\in PSH(\ti{Y},q^*\omega_{FS,p}/m), w\leq -\nu_*\mu^*\psi\}
\end{split}\]
and taking the upper-semicontinuous regularization and using the claim above gives
$\mu^*G=\mu^*\psi+\nu^*F,$ which completes the proof.
\end{proof}

Using Proposition \ref{prop51} we can see that $F$ is continuous on a Zariski open subset of $\ti{Y}$, using the following argument. Let $g: Y'\rightarrow \ti{Y}$ be a resolution of the singularities of $\ti{Y}$. Then we have:
\[
g^* F = \sup\{w\ |\ w\in PSH(Y',g^*q^*\omega_{FS,p}/m), w\leq -g^*\nu_*\mu^*\psi\}^*.
\]
\reducespace{Note that $g^*q^*\omega_{FS,p}/m$ is semi-positive and big, and that $-g^*\nu_*\mu^*\psi$ is continuous off of $g^{-1}(\nu(\mu^{-1}(\psi^{-1}(-\infty))))$,} where it is unbounded. Using the trick in \cite{McC}, we can replace the obstacle $-g^*\nu_*\mu^*\psi$ with a globally continuous obstacle $h$ without changing $g^*F$. Now, approximate $h$ uniformly by smooth functions $h_j$. It is easy to see that the envelopes:
\[
F_j := \sup\{w\ |\ w\in PSH(Y',g^*q^*\omega_{FS,p}/m), w\leq h_j\}^*.
\]
converge uniformly to $g^*F$. But then by \cite{Be}, the $F_j$ are continuous away from the non-K\"ahler locus of $g^*q^*\omega_{FS,p}/m$ (a proper Zariski closed subset, see e.g. \cite{BEGZ}), so we are done.

\begin{remark}
One is naturally led to wonder about what the optimal regularity of $G$ is. The sharp $C^{1,1}$ regularity (on a Zariski open subset) of envelopes of the form \eqref{en1} has been recently obtained in \cite{CZ,To3} in K\"ahler classes and in \cite{CTW} in nef and big classes (see also \cite{Be2,Be,BD}) when the obstacle is smooth (or at least $C^{1,1}$), but in our case the regularity of $-\nu_*\mu^*\psi$ does not seem to be very good, especially near the points where $\nu$ is not a submersion.

On the other hand, the first-named author \cite{McC} has very recently obtained $C^{1,1}$ regularity (on a Zariski open subset) of envelopes with prescribed analytic singularities, which include those of the form \eqref{en2}, generalizing results in \cite{RWN2} in the case of line bundles. In our situation, the results of \cite{McC,RWN2} do not apply since in \eqref{en2} the functions $u$ and $\psi$ are both $\omega$-psh (while for these results one would need them to be quasi-psh with respect to two different $(1,1)$-forms such that the cohomology class of their difference is big). Moreover, the main result of \cite{McC} also allows for $u$ and $\psi$ being both $\omega$-psh, but then needs the condition that the total mass of the non-pluripolar Monge-Amp\`ere operator of $\psi$ be strictly positive. This is obviously not the case in our situation however, by Theorem \ref{zeroma}.
\end{remark}
\begin{remark}
One possibly interesting approach to studying higher regularity of functions $v\in \mathcal C$ which are already continuous on $X\setminus V$ is the following. Suppose $\sup_X v = 0$. Fix an $M > 0$ and let $\Omega$ be the open set $\Omega := \{v < -M\}$. Then one can easily show using the comparison principle and Theorem \ref{zeroma} that we  have:
\[
\max\{ v, -M\} = V_\Omega - M,
\]
where here $V_\Omega$ is the global (Siciak) extremal function for $\Omega$. In particular, one sees that $\Omega$ is regular. There is then a well-developed theory about H\"older continuous regularity for such functions (the so called HCP property), see e.g. \cite{Sic}. It may be possible to use this theory to study $G$, if one can first show that it is continuous in at least a neighborhood of $V$. Another possibility may be to study regularity of the boundary of $\Omega$ -- see the very end of \cite{McC}.
\end{remark}
\begin{remark}
On can also naturally ask whether the function $\rho$ (and therefore also its singularity type $\psi$) in Theorem \ref{main} is actually independent of the choice of subsequence $t_i$, and also how regular $\rho$ is on $X\backslash V$. Our guess is that $\rho$ is indeed uniquely determined, and is smooth on $X\backslash V$. These properties would both follow if one could show that the map $\Phi:X\dashrightarrow Y$ is independent of the chosen subsequence, and that the corresponding function $u$ on $\ti{Y}$ given by Lemma \ref{parr} and Proposition \ref{corr} which satisfies
$$\mu^*\rho=\mu^*\psi+\nu^*u,$$
actually solves a suitable complex Monge-Amp\`ere equation on $\ti{Y}$. In a related setting of Calabi-Yau manifolds fibered over lower-dimensional spaces, such a limiting equation after collapsing the fibers was obtained by the second-named author in \cite[Theorem 4.1]{To0}.
\end{remark}
\begin{remark}
Lastly, we can also ask whether the limit $\rho$ (if it is unique) is necessarily equal to the pluricomplex Green's function $G$ up to addition of a constant. By remark \ref{uniq} this is the case if the rational map $\Phi$ is constant, so that $Y$ is a point. In general though this seems rather likely false.
\end{remark}

\bibliographymark{References}

\providecommand{\bysame}{\leavevmode\hbox to3em{\hrulefill}\thinspace}
\providecommand{\arXiv}[1]{\href{https://arxiv.org/abs/#1}{arXiv:#1}}
\providecommand{\MR}{\relax\ifhmode\unskip\space\fi MR }
\providecommand{\MRhref}[2]{%
  \href{http://www.ams.org/mathscinet-getitem?mr=#1}{#2}
}
\providecommand{\href}[2]{#2}

\end{document}